\documentclass[a4paper,11pt]{article}

 \usepackage{ graphicx, cite, amsfonts, amssymb, amsbsy, amsmath, amsthm}
\usepackage{enumerate}

\newtheorem{theorem}{Theorem}[section]
\newtheorem{lemma}[theorem]{Lemma}

\newtheorem{proposition}[theorem]{Proposition}
\newtheorem{corollary}[theorem]{Corollary}
\numberwithin{equation}{section}
\newtheorem{example}[theorem]{Example}
\numberwithin{equation}{section}
 
\begin{document}

\title{ \textbf{One-Generator Quasi-Abelian Codes Revisited}\thanks{This research  is    supported by the DPST Research Grant 005/2557 and  the Thailand Research Fund  under Research Grant TRG5780065.} }

\author{Somphong Jitman\thanks{Department of Mathematics, Faculty of Science, Silpakorn University, Nakhon Pathom 73000, Thailand. Email: sjitman@gmail.com}, Patanee Udomkavanich\thanks{Department of Mathematics and Computer Science,
			Faculty of Science, Chulalongkorn University,   Bangkok 10330,  Thailand. Email:pattanee.u@chula.ac.th}}

\maketitle

\begin{abstract}
The class of $1$-generator quasi-abelian codes over finite fields is revisited.   Alternative and explicit  characterization and enumeration of    such codes are given.   An algorithm to find  all $1$-generator quasi-abelian codes is provided.  Two $1$-generator quasi-abelian  codes whose minimum distances are  improved from  Grassl's online table   are presented. 

\textbf{Keywords.} group algebras, quasi-abelian codes, minimum distances, $1$-generator.
\\
\textbf{2010 AMS Classification.} 94B15,  94B60, 16A26
\end{abstract}


\section{Introduction}

As a family  of codes with good parameters, rich algebraic  structures, and  wide ranges of applications (see \cite{LF2001}, \cite{LS2001}, \cite{LS2005}, \cite{LS2003},  \cite{PZ2007}, \cite{S2004} , and references therein),  quasi-cyclic codes have been studied for a half-century. Quasi-abelian codes, a generalization of  quasi-cyclic codes,  have been introduced in \cite{W1977} and extensively  studied in  \cite{JL2014}. 

Given finite abelian groups $H\leq G$ and a finite field  $\mathbb{F}_q$, an  {\em $H$-quasi-abelian code} is defined to be   an $\mathbb{F}_q[H]$-submodule of   $\mathbb{F}_q[G]$.  Note that $H$-quasi-abelian codes are not only  a  generalization of  quasi-cyclic codes (see \cite{LF2001}, \cite{LS2001},  \cite{JL2014}, and \cite{W1977}) if $H$ is cyclic but also of  abelian codes  (see \cite{Be1967_2} and \cite{Be1967})  if $G=H$,    and  of  cyclic codes (see \cite{MS1977}) if $G=H$ is cyclic.  The characterization and enumeration of quasi-abelian codes have been established in \cite{JL2014}. An  $H$-quasi-abelian code  $C$ is said to be of  {\em $1$-generator} if $C$ is a cyclic $\mathbb{F}_q[H]$-module. Such a code can be viewed as a generalization    of   $1$-generator quasi-cyclic codes which are more frequently studied and applied (see \cite{PZ2007}, \cite{S2004}, and \cite{LS2005}).  Analogous to the case of  $1$-generator quasi-cyclic codes,         the number of    $1$-generator quasi-abelian codes has been determined in \cite{JL2014}. However, an explicit construction and an algorithm to determine  all $1$-generator quasi-abelian codes have not been well studied.

In this paper,  we give an alternative discussion on the algebraic structure of $1$-generator quasi-abelian codes and  an algorithm to find  all $1$-generator quasi-abelian codes.  Examples of new codes derived from $1$-generator quasi-abelian codes are presented

The paper is organized as follows. In Section~\ref{sec:pre}, we recall some notations and basic results.   An alternative discussion on the algebraic structure of $1$-generator quasi-abelian codes is given in  Section~\ref{sec:1gen} together with an algorithm to find   all $1$-generator quasi-abelian codes and the number of such codes.  Examples of new codes derived from $1$-generator quasi-abelian codes are presented in Section \ref{sec5}.

\section{Preliminaries}\label{sec:pre}
Let $\mathbb{F}_q$  denote  a finite field of order $q$ and let $G$ be a finite abelian group of order $n$, written additively. Denote by    $\mathbb{F}_q[G]$   the {\it group ring} of
$G$ over~$\mathbb{F}_q$. The elements in $ \mathbb{F}_q[G]$ will be written as $\sum_{g\in G}\alpha_{{g }}Y^g $,
where $ \alpha_{g }\in \mathbb{F}_q$.  The addition and the multiplication in $ \mathbb{F}_q[G]$ are  given as in the usual polynomial rings over $\mathbb{F}_q$ with the indeterminate $Y$, where the indices are computed additively in $G$. We note that  $Y^0=1$ is the identity of $ \mathbb{F}_q[G]$,  where $1$ is the identity in $ \mathbb{F}_q$ and  $0$ is the identity of $G$.

   Given a ring $\mathcal{R}$, a linear code of length $n$ over $\mathcal{R}$ refers to a submodule of the $\mathcal{R}$-module $\mathcal{R}^n$.  A {\em linear code}  in $\mathbb{F}_q[G]$ refers to an $\mathbb{F}_q$-subspace $C$ of $\mathbb{F}_q[G]$. This can be viewed as  a linear code  of length $n$ over $\mathbb{F}_q$ by indexing the $n$-tuples by the elements in $G$. 
The {\em Hamming weight} ${\rm wt}(\boldsymbol{u})$ of $\boldsymbol{u}=\sum_{g\in G}u_{{g }}Y^g \in \mathbb{F}_q[G]$  is defined to  be the number of nonzero term $u_g$'s in $\boldsymbol{u}$.  The {\em minimum Hamming distance} a  code   $C$  is defined by ${\rm d}(C):=\min\{{\rm wt}(\boldsymbol{u})\mid \boldsymbol{u}\in C, \boldsymbol{u}\neq 0\}$.   A code $C$ is referred  to as an $[n,k,d]_q$ code if $C$  has $\mathbb{F}_q$-dimension $k$ and minimum Hamming distance $d$.

 Given a subgroup $H$ of $G$, 
a code $C$ in  $\mathbb{F}_q[G]$ is called an {\em $H$-quasi-abelian code} if $C$ is an $\mathbb{F}_q[H]$-module, i.e.,   $C$ is closed under the multiplication by the elements in $\mathbb{F}_q[H]$. Such  a code will be  called a {\em quasi-abelian code} if $H$ is not specified or where it is clear in the context. An $H$-quasi-abelian code  $C$ is said to be of  {\em $1$-generator} if $C$ is a cyclic $\mathbb{F}_q[H]$-module.  

Assume that $H\leq G$ such that $|H|=m$  and the index $[G:H]=\frac{n}{m}=l$. Let $\{\mathfrak{g}_1,\mathfrak{g}_2,\dots,\mathfrak{g}_l\}$ be a fixed set of representatives of the cosets of $H$ in $G$. Let  ${R}:=\mathbb{F}_q[H]$.   Define $\Phi: \mathbb{F}_q[G] \to {R}^l$ by

\begin{align}\label{mphi}
	\Phi(\sum_{ h\in   H}\sum_{i=1}^l\alpha_{h+\mathfrak{g}_i}Y^ {h+\mathfrak{g}_i})= (\boldsymbol{\alpha}_1(Y),\boldsymbol{\alpha}_2(Y),\dots, \boldsymbol{\alpha}_l(Y)), 
\end{align}
where $\boldsymbol{\alpha}_i(Y) =	 \sum_{h\in H}\alpha_{h+\mathfrak{g}_i}Y^h\in  {R},$
for all $i=1,2,\dots,l$. It is not difficult to see that $\Phi$ is an $R$-module isomorphism, and hence, the next lemma follows.
\begin{lemma}
	The map $\Phi$ induces a one-to-one correspondence between $H$-quasi-abelian codes  in $\mathbb{F}_q[G]$ and  linear codes  of length $l$ over  ${R}$.
\end{lemma}

Throughout, assume that $\gcd(q,|H|)=1$, or equivalently,  $\mathbb{F}_q[H]$ is semisimple.  Following \cite[Section 3]{JL2014}, the group ring $R=\mathbb{F}_q[H]$ is decomposed as follows.

For each $h\in H$, denote by   ${\rm ord}(h)$ the order of $h$ in $H$. 
	A {\it $q$-cyclotomic class}   of $H$ containing $h\in H$, denoted by $S_q(h)$, is defined to be the set
	\begin{align*}
		S_q(h):=\{q^i\cdot h \mid i=0,1,\dots\} =\{q^i\cdot h \mid 0\leq i\leq \nu_h \}, 
	\end{align*}
	where $q^i\cdot h:= \sum_{j=1}^{q^i}h$ in $H$ and  $\nu_h$ is the multiplicative order of $q$ in $\mathbb{Z}_{{\rm ord}(h)}$. 

	An {\em idempotent} in a ring ${R}$ is a non-zero element $e$ such that $e^2=e$. An idempotent $e$  is said to be {\em primitive} if for every other idempotent $f$, either $ef=e$ or $ef=0$.   The  primitive idempotents in ${R}$ are induced by  the $q$-cyclotomic classes of $H$  (see \cite[Proposition II.4]{DKL2000}).    Every idempotent  $e$  in $R$ can be viewed as a unique  sum of primitive idempotents in $R$.  The $\mathbb{F}_q$-{\em dimension} of an idempotent $e\in R$ is defined to be the $\mathbb{F}_q$-dimension of $Re$.

Form \cite[Subsection 3.2]{JL2014},  ${R}:=\mathbb{F}_q[H]$ can be decomposed as 
\begin{align*}
R=Re_1+Re_2+\cdots+Re_s,
\end{align*}
	where $e_1,e_2,\dots, e_s$ are the primitive idempotents in $R$.   Moreover, every ideal in $R$ is of the form $Re$, where $e$ is an idempotent in $R$.

\section{$1$-Generator Quasi-Abelian Codes}\label{s:1-gen}\label{sec:1gen}


In \cite{JL2014}, characterization and enumeration of $1$-generator $H$-quasi-abelian codes in $\mathbb{F}_q[G]$.   
In this section, we give  alternative   characterization and enumeration of such codes. The  characterization in Subsection 3.1 allows us to express an algorithm to find all  $1$-generator $H$-quasi-abelian codes in $\mathbb{F}_q[G]$ in Subsection 3.2.

Using the $R$-module isomorphism $\Phi$ defined in (\ref{mphi}), to study $1$-generator $H$-quasi-abelian codes in $\mathbb{F}_q[G]$,  it suffices to consider  cyclic $R$-submodules $R\boldsymbol{a}$, where  $\boldsymbol{a}=(a_1,a_2,\dots,a_l)\in R^l$.

For each  $\boldsymbol{a}=(a_1,a_2,\dots,a_l)\in R^l$,  there exists a unique idempotent $e\in R$ such that $Re=Ra_1+Ra_2+\dots+Ra_l$. 
The element  $e$ is called an {\em idempotent generator element} for $R\boldsymbol{a}$. An idempotent $f\in R$ of largest $\mathbb{F}_q$-dimension such that  
\[f\boldsymbol{a}=0  \]
is called an {\em idempotent check element} for $R\boldsymbol{a}$.

Let $S=\mathbb{F}_{q^l}[H]$, where $\mathbb{F}_{q^l}$ is an  extension field of $\mathbb{F}_{q}$ of degree $l$. Let  $\{\alpha_1,\alpha_2, \dots, \alpha_l\}$ be  a fixed basis of $\mathbb{F}_{q^l}$ over $\mathbb{F}_q$.
Let $\varphi: R^l \to S$ be an $R$-module isomorphism defined by
\[\boldsymbol{a}=(a_1,a_2,\dots,a_l)\mapsto A=\sum_{i=1}^l \alpha_i a_i.\]
Using the map $\varphi$,  the code $R\boldsymbol{a}$ can be regarded as an $R$-module $RA$ in $S$.

\begin{lemma}[{\cite[Lemma 6.1]{JL2014}}] \label{lem:e+f=1}
	If $e$ and $f$ are  idempotent generator and  idempotent check elements of $R\boldsymbol{a}$, respectively,  then 
	\[e+f=1\]
	and 
	\[{\rm dim}_{\mathbb{F}_q}(R\boldsymbol{a})= {\rm dim}_{\mathbb{F}_q}(Re) = m-  {\rm dim}_{\mathbb{F}_q}(Rf).\]
\end{lemma}

For a ring $\mathcal{R}$, denote by $\mathcal{R}^*$ and $\mathcal{R}^\times$ the set of non-zero elements and the group of units of $\mathcal{R}$, respectively.

In order to enumerate and determine all   $1$-generator $H$-quasi-abelian codes  in $\mathbb{F}_q[G]$, we need the following results.
\begin{lemma}\label{lem:Ra=Rb}
Let $e$ be the  idempotent generator of $R\boldsymbol{a}$ and let $\boldsymbol{b}\in R^l$. Let $A=\varphi(\boldsymbol{a}) $ and $B=\varphi(\boldsymbol{b})$. Then  $R\boldsymbol{a}=R\boldsymbol{b}$ if and only if there exists  $u\in (Re)^\times$  such that    $\boldsymbol{b}=u\boldsymbol{a}$. 

Equivalently,  $RA=RB$ if and only if there exists  $u\in (Re)^\times$  such that  $B=uA$. 
\end{lemma}
\begin{proof}
Assume that $R\boldsymbol{a}=R\boldsymbol{b}$. 
Then $\boldsymbol{b}=v\boldsymbol{a}$  for some $v\in R$. Let $u=ve\in Re$.
Note that $a_i=r_ie$, where $r_i\in R$, for all $i=1,2,\dots,l$. Then $ua_i=ver_ie=v(r_ie)=va_i=b_i$ for all $i=1,2,\dots,l$. Hence, $\boldsymbol{b}=u\boldsymbol{a}$ and 
\begin{align*}
	Re&=Rb_1+Rb_2+\dots+Rb_l= v(Ra_1+Ra_2+\dots+Ra_l)=vRe=(ve)Re=uRe.
\end{align*}
Since  $u\in Re$ and $Re=uRe$,  we have $u\in (Re)^\times$.

Conversely, assume that  there exists    $u\in (Re)^\times$ such that   $\boldsymbol{b}=u\boldsymbol{a}$.
Then $R\boldsymbol{b}=Ru\boldsymbol{a}\subseteq R\boldsymbol{a}$.  We need to show that 
${\rm dim}_{\mathbb{F}_q}(R\boldsymbol{a})={\rm dim}_{\mathbb{F}_q}(R\boldsymbol{b})$. Let $e^\prime $ be an idempotent generator of $R\boldsymbol{b}$. We have  
\begin{align*}
	Re^\prime&=Rb_1+Rb_2+\dots+Rb_l\\
	         &=R(ua_1)+R(ua_2)+\dots+R(ua_l)\\
	  		 &=u(Ra_1+Ra_2+\dots+Ra_l)\\
	         &=u(Re)=Re
	\end{align*} 
	since $u\in (Re)^\times$.
Hence, by Lemma~\ref{lem:e+f=1}, we have ${\rm dim}_{\mathbb{F}_q}(R\boldsymbol{a})={\rm dim}_{\mathbb{F}_q}(Re)={\rm dim}_{\mathbb{F}_q}(Re^\prime) ={\rm dim}_{\mathbb{F}_q}(R\boldsymbol{b})$. Therefore,  $R\boldsymbol{b}= R\boldsymbol{a}$.
\end{proof}
 
\subsection{Enumeration of $1$-Generator Quasi-Abelian Codes}

First, we focus on the number of  $1$-generator $H$-quasi-abelian codes of a given idempotent generator. Then, the number of $1$-generator $H$-quasi-abelian codes in  $\mathbb{F}_q[G]$ follows.

\begin{proposition}\label{prop:1g-Se}
Let $\{e_1,e_2,\dots, e_r\}$ be a set of primitive idempotents of $R$ and $e=e_1+e_2+\dots+e_r$.
Then the following statements hold.
\begin{enumerate}[$i)$]
	\item  $e_1,e_2,\dots, e_r$ are pairwise orthogonal (non-zero) idempotents of $Se$.
	\item  $e_j$ is the identity of $Se_j$ for all $j=1,2,\dots,r$.
	\item $e$ is the identity of $Se$.
	\item $Se=Se_1\oplus Se_2\oplus\dots\oplus Se_r$.
\end{enumerate}
\end{proposition}
\begin{proof}
	For  $i)$,   it is clear that  $e_1,e_2,\dots, e_r$ are  pairwise orthogonal (non-zero) idempotents in $S$. They are in $Se$ since $e_j=e_je\in Se$ for all $j=1,2,\dots,r$.
	The statements $ii)$ and $iii)$ follow since   $se_j=se_j^2=(se_j)e_j$ for all $se_j\in Se_j$ and $se=se^2=(se)e$ for all $se\in Se$. The last statement can be verified using $i)$.
\end{proof}

\begin{corollary}\label{cor:1g-Re}
Let $\{e_1,e_2,\dots, e_r\}$ be a set of primitive idempotents of $R$ and $e=e_1+e_2+\dots+e_r$.
Then the following statements hold.
\begin{enumerate}[$i)$] 
	\item $e_1,e_2,\dots, e_r$ are pairwise orthogonal (non-zero) idempotents of $Re$.
	\item $e_j$ is the identity of $Re_j$ for all $j=1,2,\dots,r$.
	\item $e$ is the identity of $Re$.
	\item $Re=Re_1\oplus Re_2\oplus\dots\oplus Re_r$, where $Re_j$ is isomorphic to an extension field of $\mathbb{F}_q$ for all $j=1,2,\dots,r$.
\end{enumerate}
\end{corollary}

Let $\Omega=\left\{\sum_{j=1}^rA_j\mid     A_j\in (Se_j)^*\right\} \subset Se.$

\begin{lemma} \label{lem:1g-1}
Let $A=\sum_{i=1}^l\alpha_ia_i\in S$, where $a_i\in R$, and  let $b\in R$.
Then $RA\subseteq Sb$ if and only if $Ra_1+Ra_2+\dots+Ra_l\subseteq Rb$.
\end{lemma}
\begin{proof}
	Assume that $RA\subseteq Sb$. Then $A=Bb$ for some $B\in S$. Write $B=\sum_{i=1}^l\alpha_ib_i$, where $b_i\in R$.
	Then $a_i=bb_i$ for all $i=1,2,\dots,l$. Hence,  we have 
	\begin{align*}
		\sum_{i=1}^l r_ia_i=\sum_{i=1}^l r_ibb_i=(\sum_{i=1}^l r_ib_i)b\in Rb.
	\end{align*}
for all $\sum_{i=1}^l r_ia_i\in Ra_1+Ra_2+\dots+Ra_l.$
	
	Conversely, it suffices to show that $A\in Sb$. Since  $Ra_1+Ra_2+\dots+Ra_l\subseteq Rb$, we have $a_i\in Rb$ for all $i=1,2,\dots,l$. Then, for each $i=1,2,\dots,l$,  there exists $r_i\in R$ such that  $a_i=r_ib$. Hence, 
	\begin{align*}
		A=\sum_{i=1}^l \alpha_i a_i= \sum_{i=1}^l \alpha_i r_ib= (\sum_{i=1}^l \alpha_i r_i)b\in Sb.
	\end{align*}
\end{proof}

\begin{lemma} \label{lem:1g-2}
Let $A=\sum_{i=1}^l\alpha_ia_i \in Se $, where $a_i\in R$.
Then $A\in \Omega$ if and only if $Re=Ra_1+Ra_2+\dots+Ra_l$.
\end{lemma}
\begin{proof}
First, we note that $RA\subseteq Se$ since $A\in Se$. Then 
	$Ra_1+Ra_2+\dots+Ra_l\subseteq Re$
by  Lemma~\ref{lem:1g-1}.
	
	Assume that $A\in \Omega$. Then $A=A_1+A_2+\dots+A_r$, where $A_j\in (Se_j)^*$. We have $Ae_j=A_j\ne 0$ for all $j=1,2,\dots,r$. Suppose that  $Ra_1+Ra_2+\dots+Ra_l\subsetneq Re$. By Corollary~\ref{cor:1g-Re}, we have $Re=Re_1\oplus Re_2\oplus\dots\oplus Re_r$. Then 
	\[Ra_1+Ra_2+\dots+Ra_l\subseteq \widehat{Re_j}=R(e-e_j)\]
 for some $j\in\{1,2,\dots,r\}$, where 
$\widehat{Re_j}=Re_1\oplus\dots \oplus Re_{j-1}\oplus Re_{j+1}\oplus\dots\oplus Re_r.$
By Lemma~\ref{lem:1g-1},  we have
	\[0\ne A_j=Ae_j \in RA \subseteq S(e-e_j),\]
a contradiction.  Therefore, $Ra_1+Ra_2+\dots+Ra_l= Re$.

Conversely, assume that  $Re=Ra_1+Ra_2+\dots+Ra_l$. Then  $RA\subseteq Se$ by Lemma~\ref{lem:1g-1}. 
Since $A\in Se$, by Theorem~\ref{prop:1g-Se}, we have $A=A_1+A_2+\dots+A_r$, where $A_j\in Se_j$ for all $j=1,2,\dots, r$.
Suppose that $A_j=0$ for some $j\in \{1,2,\dots,r\}$. Then $RA=\widehat{RA_j}\subseteq \widehat{Se_j}=S(e-e_j) $.  By Lemma~\ref{lem:1g-1}, we have 
\[Re=Ra_1+Ra_2+\dots+Ra_l\subseteq R(e-e_j) \]  
which is a contradiction. Hence, $A_j\in (Se_j)^*$ for all $j=1,2,\dots,r$.
\end{proof}

\begin{corollary}\label{cor:1g-1}
Let $A=\sum_{i=1}^l\alpha_ia_i \in Se_j $, where $a_i\in R$.
Then $A\in (Se_j)^*$ if and only if $Re_j=Ra_1+Ra_2+\dots+Ra_l$.
\end{corollary}

Let $k_j$ denote the  $\mathbb{F}_q$-dimension of $e_j$.   Then $Re_j$ is isomorphic to a finite field of $q^{k_j}$ elements. 

For $j=1,2,\dots,r$, define an equivalence relation on $(Se_j)^*$ by
\begin{align*}A\sim B \iff \exists u\in (Re_j)^\times \text{ such that } A=uB.
\end{align*}

For $A\in (Se_j)^*$, denote by $[{A}]$ the equivalence class of $A$ and let $[{(Se_j)^*}]=\{[{A}]\mid A\in (Se_j)^*\}.$

\begin{lemma}For $A\in (Se_j)^*$, we have $|[{A}]|=q^{k_j}-1$.
\end{lemma}
\begin{proof}
	Define $\rho : (Re_j)^\times\to [{A}]$, 
	\[u\mapsto uA.\]
From the definition of $\sim$, $\rho$ is a well-defined surjective map.
For $u_1,u_2\in (Re_j)^\times$, if $u_1A=u_2A$, then $(u_1-u_2)A=0$. Write $A=\sum_{i=1}^l{\alpha_ia_i}$, where $a_i\in R$. 
Then $a_i(u_1-u_2)=0$ for all $i=1,2,\dots,l$. Since $A\in (Se_j)^*$, by Corollary~\ref{cor:1g-1}, we can write $e_j =\sum_{i=1}^{i}r_ia_i$, where  $r_i\in R$.  Hence, 
\begin{align*}
e_j(u_1-u_2)=(\sum_{i=1}^{i}r_ia_i)(u_1-u_2) =\sum_{i=1}^{i}r_ia_i(u_1-u_2)=0\in Re_j.
\end{align*} 
Since $e_j$ is the identity of $Re_j$, it follows that $u_1=u_2\in (Re_j)^\times.$
Therefore, $\rho$ is a bijection.
\end{proof}

\begin{corollary}	  For each $i=1,2,\dots,r$, we have 
	\[ |[{(Se_j)^*}]|=\frac{|(Se_j)^*|}{|[{A}]|}=\frac{q^{lk_j}-1}{q^{k_j}-1}.\]
Let $[{\Omega}]=\displaystyle\prod_{j=1}^r [{(Se_j)^*}].$
  Then  $|[{\Omega}]|=\displaystyle\prod_{j=1}^r \frac{q^{lk_j}-1}{q^{k_j}-1}.$
\end{corollary}

The number of   $1$-generator quasi-abelian  codes sharing a idempotent has been  determined in \cite[Corollary 6.1]{JL2014}. Here, an alternative proof using a different technique is provided. 
\begin{theorem}\label{num}
	Let $\mathfrak{C}$ denote the set of all  $1$-generator $H$-quasi-abelian  codes in  $\mathbb{F}_q[G]$ with idempotent generator  $e$. 
	Then there exists a one-to-one correspondence between $[{\Omega}]$ and $\mathfrak{C}$. Hence, the number of $1$-generator quasi-abelian  codes having $e$ as  their idempotent generator  is
	\[\prod_{j=1}^r \frac{q^{lk_j}-1}{q^{k_j}-1}.\]
\end{theorem}
\begin{proof}
	Define $\sigma: [{\Omega}]\to \mathfrak{C}$,
	\[([{A_1}],[{A_2}],\dots,[{A_r}]) \mapsto R\boldsymbol{a},\]
	where $A:=A_1+A_2+\dots+A_r\in Se$ is viewed as $A=\sum_{i=1}^l\alpha_ia_i$ and $\boldsymbol{a}:=(a_1,a_2,\dots,a_l)$.
	
	Since $A_j\in (Se_j)^*$ for all $j=1,2,\dots, r$, we have $A\in \Omega$. Then $Re=Ra_1+Ra_2+\dots+Ra_l$ by Lemma~\ref{lem:1g-2}, and hence, $R\boldsymbol{a}$ is a $1$-generator quasi-abelian  code with idempotent generator $e$, i.e., $R\boldsymbol{a}\in \mathfrak{C}$. 
	
	For  $([{A_1}],[{A_2}],\dots,[{A_r}])=([{B_1}],[{B_2}],\dots,[{B_r}])\in [{\Omega}]$, there exists $u_j\in (Re_j)^\times$ such that $A_j=u_jB_j$ for all $j=1,2,\dots,r$. Let   $u:=u_1+u_2+\dots+u_r$.
	Then 
	\[ u(u_1^{-1}+u_2^{-1}+\dots+u_r^{-1})=e_1+e_2+\dots+e_r=e\]
	 is the identity of $Re$ (see  Corollary~\ref{cor:1g-Re}), where $u_j^{-1}$ refers to the inverse of $u_j$ in $Re_j$. Hence, $u$ is a unit in  $(Re)^\times$.
  Let $A:=\sum_{j=1}^rA_j$ and $B:=\sum_{j=1}^rB_j$. Then \[A=\sum_{j=1}^ru_jB_j=uB.\]  Hence, $R\boldsymbol{a}=R\boldsymbol{b}$ by Lemma~\ref{lem:Ra=Rb}. Therefore, $\sigma$ is a well-defined map.
	
		For  $([{A_1}],[{A_2}],\dots,[{A_r}]),([{B_1}],[{B_2}],\dots,[{B_r}])\in [{\Omega}]$, if $R\boldsymbol{a}=R\boldsymbol{b}$, then, by Lemma~\ref{lem:Ra=Rb},  there exists $u\in (Re)^\times$ such that 
	$A=uB$. Then  $A_j=uB_j=ue_jB_j$ since $e_j$ is the identity of $Se_j$  by Proposition~\ref{prop:1g-Se}. 
	Since $A_j\in (Se_j)^*$, $ue_j$ is a non-zero in $Re_j$ which is a finite field. Thus  $ue_j$ is a unit in $(Re_j)^\times$.
Hence, \[([{A_1}],[{A_2}],\dots,[{A_r}])=([{B_1}],[{B_2}],\dots,[{B_r}])\]
 which implies that $\sigma$ is an injective map.

To verify that $\sigma$ is surjective, let $R\boldsymbol{a}\in \mathfrak{C}$, where $\boldsymbol{a}=(a_1,a_2,\dots,a_l)$. Then $Re=Ra_1+Ra_2+\dots+Ra_l$. Hence, by Lemma~\ref{lem:1g-2}, we conclude that \[A:=\sum_{i=1}^l\alpha_i a_i \in \Omega.\] Write $A= \sum_{j=1}^rA_j$, where $A_j\in (Se_j)^*$. Then  $([{A_1}],[{A_2}],\dots,[{A_r}])\in [{\Omega}],$  and hence, \[\sigma(([{A_1}],[{A_2}],\dots,[{A_r}]))= R\boldsymbol{a}.\]
\end{proof}

\subsection{The Generators for $1$-Generator Quasi-Abelian Codes} 

In this subsection, we establish an algorithm to find all $1$-generator $H$-quasi-abelian codes in $\mathbb{F}_q[G]$. Note that every idempotent in $R:=\mathbb{F}_q[H]$ can be written as a sum of primitive idempotents in $R$. Hence, it is sufficient to study   $H$-quasi-abelian codes for a given idempotent generator.

Let $e=e_1+e_2+\dots+e_r$ be an idempotent in $R$, where, for $j=1,2,\dots,r$, $e_j$ is  the  primitive idempotent in $R$ induced by  a $q$-cyclotomic class $S_q(h_j)$ for some $h_j\in H$. 

For each $j=1,2,\dots,r$, assume that $e_j$ is decomposed as 
\[e_j=e_{j1}+e_{j2}+\dots+e_{js_j},\] 
where, for each $i=1,2,\dots,s_j$, $e_{ji}$ is the  primitive idempotent in $S$ defined corresponding to  a $q^l$-cyclotomic class $S_{q^l}(h_{ji})$ for some $h_{ji}\in S_q(h_j)$.

Note that all the elements in $S_q(h_j)$ have the same order. Hence, the $q^l$-cyclotomic classes $S_{q^l}(h_{ji})$, for $1\leq i=1\leq s_j$, have the same size.  Without loss of generality, we assume that $e_{j1}$ is defined  corresponding to $S_{q^l}(h_j)$.   
For each $j=1,2,\dots,r$, let $k_j$ and $d_j$ denote the $\mathbb{F}_q$-dimension of $e_j$ and the $\mathbb{F}_{q^l}$-dimension of $e_{j1}$, respectively. 
Then $k_j$ and $d_j$ are the smallest positive integers such that 
\[q^{k_j}\cdot h_j=h_j \text{ and } q^{ld_j}\cdot h_j=h_j.\]
Then $k_j\mid ld_j$ which implies that $\frac{k_j}{\gcd(l,k_j)}\mid d_j$.  Since $q^{l\frac{k_j}{\gcd(l,k_j)}}\cdot h_j= q^{k_j\frac{l}{\gcd(l,k_j)}}\cdot h_j=h_j$,  we have $d_j \mid \frac{k_j}{\gcd(l,k_j)}$. It follows that $d_j= \frac{k_j}{\gcd(l,k_j)}$. Hence,  $e_{ji}$'s  have the same $q^l$-size $d_j= \frac{k_j}{\gcd(l,k_j)}$ and $s_j=\gcd(l,k_j)$. 

Using  arguments similar to those in the proof of Proposition~\ref{prop:1g-Se}, we conclude the following result.
 
\begin{proposition} \label{prop:1g-Sej}
Let $\{e_1,e_2,\dots, e_r\}$ be a set of primitive idempotents of $R$. Assume that   $e_j=e_{j1}+e_{j2}+\dots+e_{js_j}$, where $e_{ji}$ is a primitive idempotent in $S$ for all $i=1,2,\dots, s_j$.
Then the following statements hold.
\begin{enumerate}[$i)$]
	\item For $j=1,2,\dots,r$, the elements $e_{j1},e_{j2},\dots, e_{js_j}$ are pairwise orthogonal (non-zero) idempotents of $Se_j$.
	\item $e_{ji}$ is the identity of $Se_{ji}$ for all $j=1,2,\dots,r$ and $i=1,2,\dots,s_j$.  
	\item   $e_j=e_{j1}+e_{j2}+\dots+e_{js_j}$ is the identity of $Se_j$  for all  $j=1,2,\dots,r$.
	\item  For $j=1,2,\dots,r$, we have $Se_j=Se_{j1}\oplus Se_{j2}\oplus\dots\oplus Se_{js_j}$, where $Se_{ji}$ is an extension field of $\mathbb{F}_q$ of order $q^{ld_j}$ for all $i=1,2,\dots,s_j$.
\end{enumerate}
\end{proposition}

\begin{theorem}\label{all-gen}  Let $j\in\{1,2,\dots,r\}$ be fixed. For $i=1,2,\dots,s_j$, let $\pi_i$ be a primitive element of $Se_{ji}$, a finite field of $q^{ld_j}$ elements. Let $L_j=\frac{q^{ld_j}-1}{q^{k_j}-1}$ and $T_j=\{\infty, 0,1,2,\dots, q^{ld_j}-2\}$. Then the elements 
	\begin{align}\label{eq:gen}
		\pi_t^{\nu_t}+	\pi_{t+1}^{\nu_{t+1}}+\dots+\pi_{s_j}^{\nu_{s_j}},
	\end{align}
	for all $1\leq t \leq s_j$, $0\leq \nu_t\leq L_j-1$, and $\nu_{t+1},\nu_{t+2},\dots, \nu_{s_j}\in T_j$, are a complete set of representatives of $[{(Se_j)^*}]$. (By convention, $\pi_i^\infty=0$.)
\end{theorem}
\begin{proof}
Note that the number of elements  in (\ref{eq:gen}) is 
\[L_jq^{ld_j(s_j-1)}+L_jq^{ld_j(s_j-2)}+\dots+L_j=\frac{q^{lk_j}-1}{q^{k_j}-1}= |[{(Se_j)^*}]|.\]
Hence, it suffices to show that the elements in (\ref{eq:gen}) are in different equivalence classes.  Let 
\[A=\pi_t^{\nu_t}+	\pi_{t+1}^{\nu_{t+1}}+\dots+\pi_{s_j}^{\nu_{s_j}} \text{ and }B=\pi_x^{\mu_x}+	\pi_{x+1}^{\mu_{x+1}}+\dots+\pi_{s_j}^{\mu_{s_j}},\]
where $0\leq \nu_t,\mu_x\leq L_j-1$, $\nu_{t+1},$ $\nu_{t+2},\dots,\nu_{s_j}\in T_j$, and $\mu_{x+1},$ $\mu_{x+2},\dots,\mu_{s_j}\in T_j$.
Assume that $[{A}]=[{B}]$. Then there exists $u\in (Re_j)^\times$ such that 
\begin{align*}\pi_t^{\nu_t}+	\pi_{t+1}^{\nu_{t+1}}+\dots+\pi_{s_j}^{\nu_{s_j}}&=A=uB=u\pi_x^{\mu_x}+	u\pi_{x+1}^{\mu_{x+1}}+\dots+u\pi_{s_j}^{\mu_{s_j}}.
\end{align*}
Since $\pi_t^{\nu_t}\in (Se_{jt})^\times$ and $u\pi_x^{\mu_x}\in (Se_{jx})^\times$, by the decomposition in Proposition~\ref{prop:1g-Sej}, $t=x$ and $\pi_t^{\nu_t}=u\pi_t^{\mu_t}\in Se_{jt}$. Then $ue_{jt}=\pi_t^{\nu_t-\mu_t}$. Since $u\in (Re_j)^\times$, we have $u^{q^{k_j}-1}=e_j$, and hence, $e_{jt}=e_{jt}e_j=\pi_t^{(\nu_t-\mu_t)(q^{k_j}-1)}.$ Since $0\leq \nu_t,\mu_t\leq L_j-1$ and $\pi_t$ has order $q^{ld_j}-1$, we conclude that $\nu_t=\mu_t$.
Hence, $ue_{jt}=e_{jt}=e_je_{jt}$ which implies $(u-e_j)e_{jt}=0$ in $Se_{jt}$.
It follows that 
\begin{align*}
	S(u-e_j)\subseteq S(e_{j1}+\dots+e_{j,t-1}+e_{j,t+1}+\dots+e_{js_j})\subsetneq Se_j.
	\end{align*}
Since $u,e_j\in Re_j$, $u-e_j\in Re_j$ and $R(u-e_j)\subsetneq Re_j$. Hence, $R(u-e_j)$ is the zero ideal, i.e., $u=e_j$. Therefore, $A=uB=e_jB=B$ since $e_j$ is the identity of $Se_j$.
\end{proof}

The following corollary now follows from Theorem~\ref{num} and Theorem~\ref{all-gen}.
\begin{corollary}\label{corgen}
	Let $\{e_1,e_2,\dots, e_r\}$ be a set of primitive idempotents of $R$ and $e=e_1+e_2+\dots+e_r$. Then all $1$-generator quasi-abelian codes having $e$ as their idempotent generator are of the form
	\[A_1+A_2+\dots+A_r,\]
	where $A_j\in (Se_j)^*$ is as defined in (\ref{eq:gen}).
\end{corollary}

Combining the results above,  we summarize an algorithm to find all $1$-generator $H$-quasi-abelian codes as in $\mathbb{F}_q[G]$ follows.

\begin{center}
\parbox{15cm}{\hrule
\smallskip
\centerline{\bf Algorithm}
 
 \hrule 
 \bigskip 
For  abelian groups  $H\leq G$ and a finite field $\mathbb{F}_q$ with $\gcd(q,|H|)=1$ and $[G:H]=l$, do the following steps.
 \begin{enumerate}[$~~~~~1.$]
 \item Compute the $q$-cyclotomic classes of $H$ in $G$.
 \item Compute the set $\{e_1,e_2,\dots,e_r\}$ of  primitive idempotents of $R=\mathbb{F}_q[H]$ (see \cite[Proposition II.4]{DKL2000}).
 \item  For each $1\leq j\leq r$,   compute a  set $B_{j}$ of a complete set of representatives   of $[(Se_{j})^*]$ (see Theorem \ref{all-gen}).
 \item Compute the  idempotents of $R$, {\em i.e.,} the set \[T=\left\{\sum_{j=1}^te_{i_j}\mid 1\leq t\leq r    \text{ and }  1\leq i_1<i_2<\dots <i_t\leq r    \right\}.\]
 \item For each $e=\sum_{j=1}^te_{i_j}\in T$, compute the $1$-generator quasi-abelian codes having $e$ as their idempotent generator of the form
	\[A_1+A_2+\dots+A_t,\]
	where $A_j\in B_{i_j}$ (see Corollary \ref{corgen}).
	
	\item   Run $e$ over all elements of $T$. Then the $1$-generator $H$-quasi-abelian codes in $\mathbb{F}_q[G]$ are obtained.
 \end{enumerate}
 \hrule
  }
  \end{center}

\begin{example}
	Let $q=2$, $G=\mathbb{Z}_3\times \mathbb{Z}_6$ and $H=\mathbb{Z}_3\times 2\mathbb{Z}_6$. Denote by $a_0:=(0,0),$ $ a_1:=(1,0),$ $ a_2:=(2,0), $ $a_3:=(0,2), $ $a_4:=(1,2), $ $a_5:=(2,2),$ $ a_6:=(0,4), $ $a_7:=(1,4),$ and $ a_8:=(2,4),$  the elements in $H$.   Then $l=[G:H]=2$ and the elements in $H$ can be partitioned into the following $2$-cyclotomic classes $S_2(a_0)=\{a_0\}$, $S_2(a_1)=\{a_1,a_2\}$, $S_2(a_3)=\{a_3,a_6\}$,   $S_2(a_4)=\{a_4,a_8\}$, and $S_2(a_5)=\{a_7,a_5\}$. From   \cite[Proposition II.4]{DKL2000}, we note that 
	\begin{align*}
		e_1=&Y^{a_0}+Y^{a_1}+Y^{a_2}+Y^{a_3}+Y^{a_4}+Y^{a_5}+Y^{a_6}+Y^{a_7}+Y^{a_8},\\
		e_2=&Y^{a_1} +Y^{a_2}+Y^{a_4}+Y^{a_5} +Y^{a_7}+Y^{a_8},\\
		e_3= &Y^{a_3}+Y^{a_4}+Y^{a_5}+Y^{a_6} +Y^{a_7}+Y^{a_8},\\
		e_4= &Y^{a_1}+Y^{a_2}+Y^{a_3} +Y^{a_4}+Y^{a_6}+Y^{a_8} ,\\
		e_5= &Y^{a_1}+Y^{a_2} +Y^{a_3}+Y^{a_5}+Y^{a_6} +Y^{a_7}
	\end{align*}
	are primitive idempotents of $R:=\mathbb{F}_2[H]$ induced by $S_2(a_0)$, $S_2(a_1)$,  $S_2(a_3)$,   $S_2(a_4)$, and $S_2(a_5)$, respectively.
	
	Let $e:=e_1+e_2+e_3$. 	It follows from Theorem~\ref{num} that the number of $1$-generator $H$-quasi abelian codes in $\mathbb{F}_2[G]$ with idempotent generator $e$ is $3\cdot 5 \cdot 5=75$.

	Let $S:=\mathbb{F}_4[H]$, where $\mathbb{F}_4=\{0,1,\alpha,\alpha^2=1+\alpha\}$. Then $e_2=e_{21}+e_{22}$ and $e_3=e_{31}+e_{32}$, where 
\begin{align*}	
	e_{21}=&Y^{a_0}+\alpha^2Y^{a_1}+\alpha Y^{a_2}+Y^{a_3}+\alpha^2 Y^{a_4}+\alpha Y^{a_5}+Y^{a_6}+\alpha^2 Y^{a_7}+\alpha Y^{a_8},\\
	e_{22}=&Y^{a_0}+\alpha Y^{a_1}+\alpha^2Y^{a_2}+Y^{a_3}+\alpha Y^{a_4}+\alpha^2 Y^{a_5}+1 Y^{a_6}+\alpha Y^{a_7}+\alpha^2 Y^{a_8},\\
	e_{31}=&Y^{a_0}+Y^{a_1}+Y^{a_2}+\alpha^2 Y^{a_3}+\alpha^2 Y^{a_4}+\alpha^2 Y^{a_5}+\alpha Y^{a_6}+\alpha Y^{a_7}+\alpha Y^{a_8},\\ 
	e_{32}=&Y^{a_0}+ Y^{a_1}+ Y^{a_2}+\alpha Y^{a_3}+\alpha Y^{a_4}+\alpha Y^{a_5}+\alpha^2 Y^{a_6}+\alpha^2 Y^{a_7}+\alpha^2 Y^{a_8}
\end{align*}
	are primitive idempotents in $S$ induced by  $4$-cyclotomic classes $\{a_1\}$, $\{a_2\}$,  $\{a_3\} $ and $\{a_6\}$, respectively. 
	
	Now, we have $k_1=1$, $k_2=k_3=2$, $d_1=d_2=d_3=1$, $s_1=1$, and $s_2=s_3=2$. It follows that $L_1=\frac{2^2-1}{2-1}=3$, $L_2=L_3=\frac{2^2-1}{2^2-1}=1$, and $T_1=T_2=T_3=\{\infty, 0, 1,2\}$.

Then $\alpha e_1$, $\alpha e_{21}$, $\alpha e_{22}$, $\alpha e_{31}$, and $\alpha e_{32}$  are primitive elements of $Se_1$, $S e_{21}$, $S e_{22}$, $S e_{31}$, and $S e_{32}$, respectively. Therefore, we have	 that 
\begin{align*}
	B_1&= \{e_1,\alpha e_1,\alpha^2 e_1 \},\\
	B_2&= \{ e_{21}, e_{21}+ e_{22}, e_{21}+ \alpha e_{22}, e_{21}+ \alpha^2 e_{22}, \ e_{22}\},\text{ and }\\        
	B_2&= \{ e_{31}, e_{31}+ e_{32}, e_{31}+ \alpha e_{32}, e_{31}+ \alpha^2 e_{32}, \ e_{32}\}
\end{align*}
are complete sets of representatives of $[{(Se_1)^*}]$, $[{(Se_2)^*}]$, and $[{(Se_3)^*}]$, respectively.
Hence, all the generators of the $75$ $1$-generator $H$-quasi abelian codes in $\mathbb{F}_2[G]$ with idempotent generator $e$ are of the form 
\[A_1+A_2+A_3,\]
where $A_i\in B_i$ for all $i=1,2,3$.
\end{example}

In order to find  permutation inequivalent $1$-generator $H$-quasi abelian codes, the following theorem is useful.
\begin{theorem} Let $H\leq G$ be finite abelian groups of index $[G:H]=l$ and let $\{\alpha^{q^i}\mid 1\leq i\leq l\}$ be a fixed basis of $\mathbb{F}_{q^l}$ over $\mathbb{F}_q$.
	If $A=\sum_{i=1}^la_i\alpha^{q^i}\in Se$, then $A$ and $A^q=\sum_{i=1}^la_i^q\alpha^{q^{i+1}}$ generate permutation equivalent $H$-quasi abelian codes (viewed in $\mathbb{F}_q[G]$) with the same idempotent generator.
\end{theorem}
\begin{proof} Let $e$ be the idempotent generator of a quasi-abelian code $RA$.
Then  
\begin{align*}
	 Ra_1^q+Ra_2^q+\dots+Ra_l^q \subseteq Ra_1+Ra_2+\dots+Ra_l=Re
	\end{align*}
 Assume that  $e=\sum_{i=1}^lr_ia_i$, where $r_i\in R$. It follows that  
\[e=e^q=\sum_{i=1}^lr_i^qa_i^q \in Ra_1^q+Ra_2^q+\dots+Ra_l^q.\] Hence, $Re= Ra_1^q+Ra_2^q+\dots+Ra_l^q$. Therefore, $A$ and $A^q$ generate codes with the same idempotent generator $e$.  

Let $\psi:R\to R$ be a ring homomorphism defined by 
\[\gamma \mapsto \gamma^q.\]
Let $\gamma=\sum_{h \in H}\gamma_hY^h$ and $\beta=\sum_{h \in H}\beta_hY^h$ be elements in $R$, where $\gamma_h$ and $\beta_h$ are elements in $\mathbb{F}_q$. If $\psi(\gamma)=\psi(\beta)$, then 
\[0=\gamma^q-\beta^q =(\gamma-\beta)^q=\sum_{h \in H}(\gamma_h-\beta_h)Y^{q\cdot h}.\] By comparing the coefficients, we have $\gamma_h=\beta_h$ for all $h\in H$, $i.e.$, $\gamma=\beta$. 
Hence, $\psi$ is a ring automorphism and 
\begin{align}\label{psi}
	R(a_l^q,a_1^q,\dots, a_{l-1}^q)&=R(\psi(a_l),\psi(a_1),\dots, \psi(a_{l-1}))\notag\\
	                               &= \Psi(R(a_l,a_1,\dots, a_{l-1}) ),
\end{align}
 where $\Psi$ is a natural extension of $\psi$ to $R^l$.

Since $\psi(\gamma)=\sum_{h \in H}\gamma_h Y^{q\cdot h}$, $\psi(\gamma)$ is just a permutation on the coefficients of $\gamma$. Hence, by (\ref{psi}), $\Psi\circ\Phi$ is a permutation on  $\mathbb{F}_q[G]$ such that  $\Phi^{-1}\left(R(a_l^q,a_1^q,\dots, a_{l-1}^q)\right)$ is permutation equivalent to $\Phi^{-1}\left(R(a_l,a_1,\dots, a_{l-1})\right)$   in $\mathbb{F}[G]$, where $\Phi$ is the $R$-module isomorphism defined in (\ref{mphi}). Therefore, the result follows since $R(a_l,a_1,\dots, a_{l-1})$ is permutation equivalent  to $R(a_1,a_2,\dots, a_{l})$.
\end{proof}

\section{Computational Results}\label{sec5}

It has been shown in \cite{J2015}  and  \cite{JL2014}  that a family of  quasi-abelian codes   contains various new and optimal codes. 
Here, we present  other $2$     new codes   from the  quasi-abelian  codes  together with  $1$ new code     obtained by    shortening of one of these  codes.

Given  an abelian group  $H=\mathbb{Z}_{n_1}\times \mathbb{Z}_{n_2}$ of order $n=n_1n_2$,
 denote by   $u=(u_0,u_1,u_2,\dots, u_{n-1})\in
\mathbb{F}_q^n$ the vector representation of   
\[ u =\sum\limits_{j=0}^{n_2-1}\sum\limits_{i=0}^{n_1-1}u_{jn_1+i}  Y^ {(i,j)} \text{ in } \mathbb{F}_q[H].\]
Let 
\begin{align}\label{cab}
 C_{(a,b)}:=\{(fa,fb)\mid f \in \mathbb{F}_q[H]\},
\end{align}
where $a$ and $b$ are elements in    $\mathbb{F}_q[H]$.
Using (\ref{cab}),  
 $2$  quasi-abelian codes whose minimum distance improves on Grassl's online table \cite{G2013} can be found. The codes $C_1$ and $C_2$ are  presented in Table \ref{T} and the generator matrices of  $C_1$ and $C_2$ are 
\[G_1=
\left[\begin{array}{c|c}
&~1~ 3~ 0~ 3~ 4~ 1~ 3~ 2~ 0~ 4~ 1~ 2~ 1~ 4~ 0~ 4~ 1~ 0~ 4~ 3~ 0~ 4~\\
&~ 1~ 3~ 4~ 4~ 3~ 1~ 4~ 0~ 2~ 4~ 1~ 3~ 0~ 2~ 2~ 4~ 3~ 1~ 1~ 3~ 4~ 0~\\
&~ 1~ 4~ 4~ 3~ 4~ 0~ 4~ 0~ 0~ 1~ 0~ 3~ 1~ 2~ 0~ 1~ 0~ 3~ 2~ 4~ 4~ 4~\\
&~ 4~ 4~ 3~ 3~ 4~ 2~ 3~ 3~ 1~ 3~ 4~ 0~ 3~ 3~ 2~ 1~ 1~ 1~ 1~ 0~ 3~ 0~\\
&~ 4~ 3~ 3~ 4~ 3~ 2~ 4~ 2~ 3~ 2~ 3~ 2~ 2~ 3~ 0~ 3~ 2~ 1~ 0~ 1~ 4~ 3~\\
&~ 4~ 4~ 2~ 4~ 4~ 1~ 4~ 1~ 2~ 4~ 2~ 1~ 4~ 0~ 0~ 1~ 1~ 2~ 0~ 4~ 0~ 4~\\
&~ 0~ 2~ 1~ 1~ 3~ 1~ 4~ 1~ 1~ 2~ 1~ 0~ 1~ 1~ 4~ 2~ 0~ 0~ 1~ 3~ 2~ 3~\\
I_{14}~& ~0~ 1~ 2~ 1~ 4~ 3~ 1~ 2~ 1~ 1~ 1~ 1~ 0~ 2~ 1~ 4~ 1~ 0~ 0~ 3~ 3~ 2~\\
&~ 0~ 1~ 1~ 2~ 1~ 4~ 3~ 1~ 2~ 1~ 0~ 1~ 1~ 4~ 2~ 1~ 0~ 1~ 0~ 2~ 3~ 3~\\
&~ 1~ 2~ 2~ 2~ 3~ 4~ 4~ 4~ 4~ 1~ 3~ 1~ 4~ 4~ 3~ 3~ 1~ 0~ 1~ 2~ 2~ 4~\\
&~1~ 2~ 3~ 1~ 4~ 0~ 2~ 2~ 4~ 3~ 4~ 0~ 4~ 1~ 2~ 2~ 0~ 1~ 1~ 3~ 3~ 2~\\
&~1~ 1~ 3~ 2~ 2~ 1~ 3~ 4~ 2~ 3~ 4~ 1~ 3~ 0~ 4~ 1~ 0~ 0~ 2~ 1~ 4~ 3~\\
&~4~ 0~ 4~ 1~ 0~ 3~ 2~ 4~ 0~ 1~ 0~ 3~ 2~ 2~ 2~ 1~ 1~ 0~ 4~ 1~ 4~ 0~\\
&~4~ 1~ 4~ 0~ 2~ 3~ 0~ 0~ 4~ 1~ 2~ 3~ 0~ 3~ 4~ 3~ 0~ 1~ 4~ 1~ 0~ 4~
\end{array}\right]
\]
and 
\[
G_2=
\left[\begin{array}{c|c}
  &~0~ 1~  0~ 4~ 4~ 0~ 0~ 1~ 4~ 4~ 0~ 4~ 1~ 3~ 2~ 3~ 3~ 1~ 1~ 3~ 3~ 2~ 0~ 1~ 4~\\
  &~4~ 4~  1~ 1~ 2~ 1~ 2~ 4~ 1~ 4~ 3~ 2~ 1~ 4~ 4~ 3~ 2~ 4~ 2~ 0~ 1~ 1~ 0~ 1~ 2~\\
  &~1~ 0~  4~ 0~ 0~ 0~ 4~ 4~ 4~ 1~ 4~ 1~ 0~ 2~ 3~ 3~ 1~ 1~ 3~ 3~ 2~ 3~ 1~ 4~ 0~\\
  &~0~ 1~  0~ 0~ 4~ 0~ 4~ 1~ 0~ 3~ 1~ 3~ 0~ 3~ 1~ 4~ 1~ 3~ 4~ 1~ 4~ 3~ 3~ 4~ 1~\\
  &~4~ 4~  0~ 0~ 0~ 0~ 1~ 1~ 4~ 3~ 3~ 4~ 1~ 4~ 3~ 1~ 4~ 1~ 3~ 0~ 3~ 1~ 3~ 0~ 1~\\
  I_{11}~&~1~ 0~  0~ 0~ 0~ 4~ 4~ 0~ 3~ 1~ 3~ 0~ 1~ 1~ 4~ 3~ 3~ 4~ 1~ 4~ 3~ 1~ 4~ 1~ 3~\\
  &~1~ 1~  4~ 0~ 4~ 0~ 4~ 3~ 2~ 1~ 0~ 0~ 4~ 1~ 3~ 1~ 2~ 3~ 3~ 2~ 3~ 4~ 2~ 4~ 2~\\
  &~4~ 0~  0~ 4~ 0~ 0~ 1~ 4~ 1~ 0~ 2~ 3~ 3~ 1~ 1~ 3~ 3~ 2~ 3~ 1~ 4~ 0~ 4~ 4~ 1~\\
  &~0~ 4~  1~ 1~ 2~ 1~ 1~ 2~ 1~ 3~ 2~ 1~ 2~ 4~ 2~ 2~ 4~ 4~ 3~ 1~ 2~ 0~ 0~ 3~ 3~\\
  &~1~ 1~  0~ 0~ 4~ 4~ 4~ 2~ 2~ 2~ 2~ 2~ 2~ 0~ 0~ 0~ 0~ 0~ 0~ 3~ 3~ 3~ 3~ 3~ 3~\\
  &~0~ 0~  1~ 1~ 1~ 1~ 1~ 2~ 2~ 2~ 4~ 4~ 4~ 1~ 1~ 1~ 1~ 1~ 1~ 1~ 1~ 1~ 4~ 4~ 4~
\end{array}\right],
\] respectively.

\begin{table}[!hbt]
	\centering
	\caption{New Codes from Quasi-Abelian   Codes}\label{T}
		\begin{tabular}{|c|r|r|l|}
\hline
name& $C_{(a,b)}$ & $H$ & $a,b$\\\hline
$C1$&$[36,14,15]_5 $&$\mathbb{Z}_3  \times \mathbb{Z}_{6}$& $a=(3, 3, 3, 0, 0, 1, 4, 3, 4, 0, 4, 4, 4, 4, 3, 0, 1, 0)$\\  
&&&$b=( 2, 4, 1, 1, 3, 3, 0, 0, 4, 4, 1, 0, 0, 1, 4, 2, 2, 4)$\\\hline
$C2$&$[36,11,18]_5 $&$\mathbb{Z}_3  \times \mathbb{Z}_{6}$& $a=( 2, 4, 4, 3, 4, 4, 3, 2, 4, 3, 4, 4, 3, 4, 2, 3, 4, 4)$\\  
&&&$b=(3, 0, 0, 0, 3, 3, 3, 0, 3, 0, 3, 0, 1, 1, 1, 1, 1, 1 )$\\ \hline
              \end{tabular}
\end{table}

 By puncturing $C_2$ at the  first coordinate, a $[35,11,17]_5 $ code can be obtained   with minimum distance improved by $1$ from Grassl's online table \cite{G2013}. All the computations are done using   MAGMA \cite{BCP1997}.

\section*{Acknowledgments}
  The authors thank San Ling   for useful discussions.


\end{document}